\documentclass[a4paper,10pt]{amsart}
\usepackage{amsmath,amssymb}
\usepackage{amscd}
\usepackage{amsthm}
\usepackage{stmaryrd}
\usepackage[dvipdfm]{graphicx}

\numberwithin{equation}{section}

\usepackage{natbib}
\bibstyle{plain}

\DeclareMathOperator{\Jac}{Jac}
\DeclareMathOperator{\Nij}{Nij}

\title[Generalized almost contact and generalized Sasakian structures]{Generalized almost contact structures and generalized Sasakian structures}
\author{Ken'ichi Sekiya }
\email{kenichisekiya.g@gmail.com}


\begin{document}
	\theoremstyle{definition}
	\newtheorem{defn}{Definition}[section]
	\newtheorem{thm}[defn]{Theorem}
	\newtheorem{prop}[defn]{Proposition}
	\newtheorem{lem}[defn]{Lemma}
	\newtheorem{cor}[defn]{Corollary}
	\newtheorem{fac}{Fact}[section]
	\newtheorem{exam}[defn]{Example}
	\newtheorem{conj}{Conjecture}
	\newtheorem{rem}[defn]{Remark}

\maketitle

\begin{abstract}
We introduce generalized almost contact structures which admit the $B$-field transformations on odd dimensional manifolds.
We provide definition of generalized Sasakain structures from the view point of the generalized almost contact structures.
We obtain a generalized Sasakian structure on a non-compact manifold which does not arise as a pair of ordinary Sasakian structures.
However we show that a generalized Sasakian structure on compact $3$-dimensional manifold is equivalent to a pair of Sasakian structures with the same metric.
Finally we extend a definition of a generalized almost contact structure.
\end{abstract}

\section{introduction}

Both generalized complex structures and generalized K\"ahler structures are structures on even dimensional manifolds.
It is natural to ask what is an analog of generalized geometry on odd dimensional manifolds.
Vaisman introduced generalized F-structures and generalized almost contact structures \cite{Vais, Vais3}.
He also defined generalized Sasakian structures from the view point of generalized K\"ahler structures.
Poon and Wade studied integrability conditions of generalized almost contact structures and gave nontrivial examples on the three-dimensional Heisenberg group and its cocompact quotients \cite{PoWa}.
Vaisman showed that a generalized Sasakian structure appears as a pair of almost contact structures \cite{Vais}.
However, examples of generalized Sasakian structures which do not arise as a pair of Sasakian structures were not known.

The purpose of this paper is to investigate generalized geometry on odd dimensional manifolds.
We introduce the new notion of generalized almost contact structures
 which includes the one in \cite{Vais}, \cite{PoWa} as special cases. 
We use two sections $E_+$ and $E_-$ of $TM\oplus T^*M$ to define
generalized almost contact structures which admit $B$-field transformations naturally.
An almost contact structure is a triple $(\varphi,\xi,\eta)$, where $\varphi$ is an endomorphism of $TM$, $\xi \in TM$ and $\eta \in T^{*}M$ which satisfies
\begin{equation*}
\eta(\xi)=1,\quad \varphi\circ\varphi=-id+\eta\otimes\xi,
\end{equation*}
where $id$ denotes the identity map of $TM$.
An almost contact structure gives rises to an almost complex structure $I$ on the cone $C(M)=M\times \mathbb{R}_{>0}$,
\begin{equation*}
I=\varphi+\eta\otimes r\frac{\partial}{\partial r}-\frac{1}{r}dr\otimes\xi,
\end{equation*}
where $r$ denotes the coordinate on $\mathbb{R}_{>0}$.
We define a generalized almost contact structure to be a triple ($\Phi,E_+, E_-)$ by replacing $\varphi$ by an endomorphism $\Phi$ of $TM\oplus T^{*}M$ and $\xi,\eta$ by sections $E_+, E_-$ of $TM\oplus T^{*}M$, respectively
which satisfy 
\begin{gather*}
\Phi+\Phi^{*}=0,\\
2\langle E_{+},E_{-}\rangle=1,\;\langle E_{\pm},E_{\pm} \rangle=0,\\
\Phi\circ\Phi=-id+E_{+}\otimes E_{-}+E_{-}\otimes E_{+}
\end{gather*}
 (see Definition 3.1
for more detail).
By an analogue to the case of almost contact structures,
we define bundle endomorphisms to construct generalized complex structures on the cone $C(M)$.
Putting
\begin{equation*}
\Psi(E_{+},E_{-})
=E_{-}\otimes r\frac{\partial}{\partial r}-r\frac{\partial}{\partial r}\otimes E_{-}+E_{+}\otimes \frac{1}{r}dr -\frac{1}{r}dr\otimes E_{+},
\end{equation*}
it follows that
\begin{equation*}
\Phi+\Psi(E_{+},E_{-})
\end{equation*}
is generalized almost complex structure on $C(M)$.
In Sasakian geometry, the Riemannian cone metric $\tilde{g}=dr^{2}+r^{2}g$ on $C(M)$ is, by definition, a K\"ahler metric.
This suggests that
\begin{equation*}
R(\Phi+\Psi(E_{+},E_{-}))R^{-1}
\end{equation*}
is more important generalized almost complex structures rather than $\Phi+\Psi(E_{+},E_{-})$ when we pursue an analogy of  Sasakian geometry, 
where $R$ denotes an element of  the special orthogonal group SO$(TM\oplus T^*M)$ given by 
$$
R(X+\alpha)=r^{-1}X+r\alpha, \qquad X\in TM, \quad \alpha\in T^*M.
$$
From the view point of generalized almost contact structures, we define a generalized Sasakian structure.
We show that on a compact connected 3-dimensional manifold a generalized Sasakian structure is equivalent to a pair of Sasakian structures with the same metric (Theorem \ref{th:GS on 3-mfds}).
We obtain a non-compact example of a generalized Sasakian structure which does not arise as a pair of Sasakian structures (Theorem \ref{ex:GS, non-Sasakian pair}).

Finally we introduce an extended definition of generalized almost contact structure.
Although generalized almost contact structures admit $B$-field transformations,
generalized almost complex structures $R(\Phi+\Psi(E_{+},E_{-}))R^{-1}$ don't admit $B$-field transformations by $2$-forms $2rdr\wedge\kappa$ $(\kappa \in T^{*}M)$.
To admit $B$-field transformations on cone, we define a generalized $f$-almost contact structure.
Then we can naturally define two transformations of generalized $f$-almost contact structures.
One is a correspondence with $B$-field transformations on cone.
We show that the other is a correspondence with a cross term of generalized Riemannian metric on cone.

{\bf Acknowledgment.}
The author would like to thank Professor R. Goto for his many valuable suggestions and comments.

\section{generalized complex structures}
\label{sec:gcpxs}

In this section we give a brief explanation of generalized complex structures.
Let $M$ be an even dimensional smooth manifold.
The space of sections of the vector bundle $TM\oplus T^{*}M\rightarrow M$ is endowed with the following $\mathbb{R}$-bilienar operations. 
\begin{itemize}
\item
A symmetric bilinear form $\langle -,- \rangle$ is defined by
\begin{equation*}
\langle X+\alpha,Y+\beta \rangle=\frac{1}{2}(\iota_{X}\beta+\iota_{Y}\alpha).
\end{equation*}

\item
The Courant bracket $\llbracket -,-\rrbracket$ is a skew-symmetric bracket,
\begin{equation*}
\llbracket X+\alpha,Y+\beta \rrbracket=[X,Y]+\mathcal{L}_{X}\beta-\mathcal{L}_{Y}\alpha-\frac{1}{2}d(\iota_{X}\beta-\iota_{Y}\alpha),
\end{equation*}
where $X,Y \in TM$ and $\alpha,\beta \in T^{*}M$.
\end{itemize}

A subbundle is Courant involutive
if the space of sections of the subbundle is closed under the Courant bracket.

\begin{defn}\cite{Gualt-d}
A generalized almost complex structure on $M$ is an endomorphism of the direct sum
$TM\oplus T^{*}M$ which satisfies two conditions,
\begin{equation*}
\mathcal{J}+\mathcal{J}^{*}=0,\qquad \mathcal{J}^{2}=-id,
\end{equation*}
where $\mathcal{J}^{*}$ is defined by $\langle \mathcal{J}A,B\rangle=\langle A,\mathcal{J}^{*}B\rangle$ for any $A,B \in \Gamma(TM\oplus T^{*}M)$.
Let $L$ be the $+\sqrt{-1}$-eigenspace of $\mathcal{J}$ in $TM\oplus T^{*}M$.
If $L$ is Courant involutive, then $\mathcal{J}$ is called a generalized complex structure.
\end{defn}

The following are well known.

\begin{lem}\cite{Gualt-d}\label{L-isot}
$L$ is a maximal isotropic subspace.
\end{lem}

\begin{prop}\cite{Gualt-d}\label{L-invo}
Let $L$ be a maximal isotropic subbundle of $TM\oplus T^{*}M$.
Then the following three conditions are equivalent:
\begin{itemize}
	\item $L$ is Courant involutive,
	\item $\Nij|_{L}=0$,
	\item $\Jac|_{L}=0$,
\end{itemize}
where $\Nij$ and $\Jac$ are given by
\begin{align*}
&\Nij(A,B,C)=\frac{1}{3}\left(\langle\llbracket A,B\rrbracket,C\rangle+\langle\llbracket B,C\rrbracket ,A\rangle+\langle\llbracket C,A\rrbracket,B\rangle\right),\\
&\Jac(A,B,C)=\llbracket \llbracket A,B\rrbracket,C\rrbracket+\llbracket \llbracket B,C\rrbracket ,A\rrbracket +\llbracket \llbracket C,A\rrbracket ,B\rrbracket,
\end{align*}
for any $A,B,C \in \Gamma(TM\oplus T^{*}M)$ .
\end{prop}

Let $B$ be a smooth $2$-form.
Then the invertible bundle map given by exponentiating $B$,
\begin{equation*}
e^{B}=
\begin{pmatrix}
1&0\\B&1
\end{pmatrix}
\colon
X+\alpha \mapsto
X+\alpha+\iota_{X}B
\end{equation*}
is orthogonal.

\begin{lem}\cite{Gualt-d}\label{int-close}
A map $e^{B}$ is an automorphism of the Courant bracket if and only if $B$ is closed, i.e. $dB=0$.
\end{lem}

\begin{defn}\cite{Gualt-d}
A generalized K\"ahler structure is a pair $(\mathcal{J}_{1},\mathcal{J}_{2})$ of commuting generalized complex structures such that $G=-\mathcal{J}_{1}\mathcal{J}_{2}$ gives a positive definite metric on $TM\oplus T^{*}M$.
\end{defn}

\begin{lem}\cite{Gualt-d}
A generalized K\"ahler metric is uniquely determined by a Riemannian metric $g$ together with a $2$-form $b$ as follows,
\begin{equation*}
G(g,b)=
\begin{pmatrix}
-g^{-1}b&g^{-1}\\
g-bg^{-1}b&bg^{-1}
\end{pmatrix}
=\begin{pmatrix}
1&0\\b&1
\end{pmatrix}
\begin{pmatrix}
0&g^{-1}\\g&0
\end{pmatrix}
\begin{pmatrix}
1&0\\-b&1
\end{pmatrix}.
\end{equation*}
\end{lem}

Let $C_{+}$ be a positive definite subbundle of $TM\oplus T^{*}M$ and $C_{-}$ a negative definite subbundle with respect to the inner product which are given by
\begin{equation*}
C_{\pm}=\{X\pm g(X,\cdot)+b(X,\cdot)\;;\;X \in TM\}.
\end{equation*}
The projection from $C_{\pm}$ to $TM$, $\mathcal{J}_{1}$ induces two almost complex structures $J_{\pm}$ on $TM$.
If both $(g,J_{+})$ and $(g,J_{-})$ are Hermitian structures, $(g,J_{\pm})$ is called a bi-Hermitian structure.

\begin{thm}\cite{Gualt-d}\label{reconst}
A generalized K\"ahler structure $(\mathcal{J}_{1},\mathcal{J}_{2})$ is equivalent to bi-Hermitian structure $(g,b,J_{\pm})$ which satisfies the following condition.
\begin{itemize}
\item For all vector fields $X,Y,Z$,
\begin{equation*}
db(X,Y,Z)=d\omega_{+}(J_{+}X,J_{+}Y,J_{+}Z)=-d\omega_{-}(J_{-}X,J_{-}Y,J_{-}Z),
\end{equation*} 
\end{itemize}
where $\omega_{\pm}(X,Y)=g(X,J_{\pm}Y)$.
\end{thm}

\section{generalized almost contact structures}\label{sgcs}

An almost contact structure on an odd dimensional manifold $M$ is a triple $(\varphi,\xi,\eta)$, where $\varphi$ is an endomorphism of $TM$, $\xi$ is a vector field, and $\eta$ is a $1$-form which satisfies
\begin{equation*}
\eta(\xi)=1,\quad \varphi\circ\varphi=-id+\eta\otimes\xi.
\end{equation*}
We replace $\varphi$ by an endomorphism $\Phi$ of $TM\oplus T^{*}M$ and $\xi,\eta$ by sections $E_{\pm}$ of $TM\oplus T^{*}M$ respectively.
We define a generalized almost contact structure:

\begin{defn}\label{gac}
A generalized almost contact structure on a smooth manifold $M$ is a triple $(\Phi,E_{+},E_{-})$, where $\Phi$ is an endomorphism of $TM\oplus T^{*}M$ and $E_{\pm}$ are sections of $TM\oplus T^{*}M$ which satisfy
\begin{gather*}
\Phi+\Phi^{*}=0,\\
2\langle E_{+},E_{-}\rangle=1,\;\langle E_{\pm},E_{\pm} \rangle=0,\\
\Phi\circ\Phi=-id+E_{+}\otimes E_{-}+E_{-}\otimes E_{+}.
\end{gather*}
\end{defn}

Let $E_{\pm}=\xi_{\pm}+\eta_{\pm}$ where $\xi_{\pm}$ are vector fields and $\eta_{\pm}$ are $1$-forms.
Then we have
\begin{gather*}
\Phi\circ\Phi=-id+
\begin{pmatrix}
\eta_{+}\otimes\xi_{-}+\eta_{-}\otimes\xi_{+}&\xi_{+}\otimes \xi_{-}+\xi_{-}\otimes \xi_{+}\\
\eta_{+}\otimes \eta_{-}+\eta_{-}\otimes\eta_{+}&\xi_{+}\otimes\eta_{-}+\xi_{-}\otimes\eta_{+}
\end{pmatrix}.
\end{gather*}

\begin{rem}
Vaisman, Poon and Wade discussed the restrictive case of $\xi_{-}=\eta_{+}=0$ \cite{PoWa,Vais,Vais3}.
However, their definition is not compatible with the $B$-field transformations.
Note that a generalized almost contact structure of Definition \ref{gac} satisfies the condition of generalized $F$-structure \cite{Vais3}.
\end{rem}

\begin{exam}\cite{PoWa}\label{ex: al-con-st}
Let $(\varphi,\xi,\eta)$ be an almost contact structure.
Then we have a generalized almost contact structure by setting
\begin{equation*}
\Phi=
\begin{pmatrix}
\varphi&0\\0&-\varphi^{*}
\end{pmatrix},
\;E_{+}=\xi,\;E_{-}=\eta,
\end{equation*}
where $(\varphi^{*}\alpha)(X)=\alpha(\varphi X),\;X \in TM,\;\alpha\in T^{*}M$.
\end{exam}
\begin{exam}\cite{PoWa}
A $(2n+1)$-dimensional manifold $M$ is a contact manifold if there exists a $1$-form $\eta$ such that
\begin{equation*}
\eta\wedge (d\eta)^{n}\neq 0
\end{equation*}
everywhere on $M$.
A $1$-form $\eta$ is called a contact $1$-form.
Then there is a unique vector field $\xi$ satisfying the two conditions
\begin{equation*}
\iota_{\xi}d\eta=0,\quad,\eta(\xi)=1.
\end{equation*}
This vector field is called the Reeb field of the contact form $\eta$.
Since $\eta$ is a contact $1$-form, the map
\begin{equation*}
\rho(X):=\iota_{X}d\eta-\eta(X)\eta
\end{equation*}
is an isomorphism from the tangent bundle to the cotangent bundle.
We define a bivector field $\pi$ by
\begin{equation*}
\pi(\alpha,\beta):=d\eta(\rho^{-1}(\alpha),\rho^{-1}(\beta)).
\end{equation*}
Then we have a generalized almost contact structure by setting
\begin{equation*}
\Phi=
\begin{pmatrix}
0&\pi\\ d\eta&0
\end{pmatrix},
\;E_{+}=\eta,\;E_{-}=\xi.
\end{equation*}
\end{exam}

\begin{lem}
Let $(\Phi,E_{\pm})$ be a generalized almost contact structure.
Then we have the following identities,
\begin{equation*}
\Phi(E_{\pm})=0
\end{equation*}
\end{lem}
\begin{proof}
Since $\Phi+\Phi^{*}=0$, we have
\begin{equation*}
\langle \Phi E_{+},E_{+}\rangle=\langle E_{+},-\Phi E_{+}\rangle =-\langle \Phi E_{+},E_{+}\rangle.
\end{equation*}
Thus it follows that
\begin{equation*}
\langle \Phi E_{+},E_{+}\rangle=0.
\end{equation*}
Then we obtain
\begin{equation}\label{l-1}
\begin{split}
0&=\Phi\circ(\Phi\circ\Phi)(E_{+})=(\Phi\circ\Phi)\circ\Phi(E_{+})\\
&=-\Phi E_{+}+2\langle E_{+},\Phi E_{+}\rangle E_{-}+2\langle E_{-},\Phi E_{+}\rangle E_{+}\\
&=-\Phi E_{+}+2\langle E_{-},\Phi E_{+}\rangle E_{+}.
\end{split}
\end{equation}
We also obtain
\begin{equation}\label{l-2}
0=\Phi\circ(\Phi\circ\Phi)\circ\Phi(E_{+})
=2\langle E_{-},\Phi E_{+}\rangle \Phi E_{+}.
\end{equation}
From (\ref{l-1}) and (\ref{l-2}), we have
\begin{equation*}
\Phi E_{+}=0.
\end{equation*}
Similarly, we have
\begin{equation*}
\Phi E_{-}=0.
\end{equation*}
\end{proof}

By a simple calculation, we get

\begin{lem}
Let $(\Phi,E_{\pm})$ be a generalized almost contact structure and $B$ a smooth $2$-form.
Then $(e^{B}\Phi e^{-B},e^{B}E_{\pm})$ is a generalized almost contact structure.
\end{lem}

By Definition \ref{gac}, we have
\begin{equation*}
\Phi^{3}+\Phi=0.
\end{equation*}
Thus $\Phi$ has three eigenvalues, namely $0,+\sqrt{-1},-\sqrt{-1}$.
The kernel of $\Phi$ is given by
\begin{equation*}
L_{E_{+}}\oplus L_{E_{-}},
\end{equation*}
where $L_{E_{\pm}}$ are line bundles generated by $E_{\pm}=\xi_{\pm}+\eta_{\pm}$, respectively.
We define
\begin{align*}
E^{(1,0)}&=\{X+\alpha-\sqrt{-1}\Phi(X+\alpha)\;;\;X \in TM, \alpha \in T^{*}M, \langle X+\alpha,E_{\pm}\rangle=0\},\\
E^{(0,1)}&=\{X+\alpha+\sqrt{-1}\Phi(X+\alpha)\;;\;X \in TM, \alpha \in T^{*}M, \langle X+\alpha,E_{\pm}\rangle=0\}.
\end{align*}
Then $E^{(1,0)}$ is $+\sqrt{-1}$-eigenbundle and $E^{(0,1)}$ is $-\sqrt{-1}$-eigenbundle.
We consider the following four different complex vector bundles,
\begin{equation}
\begin{split}
L^{+}=L_{E_{+}}\oplus E^{(1,0)},&\quad \overline{L^{+}}=L_{E_{+}}\oplus E^{(0,1)},\\
L^{-}=L_{E_{-}}\oplus E^{(1,0)},&\quad \overline{L^{-}}=L_{E_{-}}\oplus E^{(0,1)}.
\end{split}
\end{equation}

The next lemma corresponds to Lemma 2.3 in \cite{PoWa}

\begin{lem}
Bundles $E^{(1,0)},E^{(0,1)},L^{\pm},\overline{L^{\pm}}$ are isotropic.
\end{lem}

\begin{proof}
Let $A,B$ are sections of $E^{(1,0)}$.
By our definition, we have $\langle A,E_{\pm}\rangle=0$.
It follows from $\Phi+\Phi^{*}=0$ that
\begin{align*}
&\langle \Phi A,\Phi B\rangle=\langle \sqrt{-1}A,\sqrt{-1}B\rangle=-\langle A,B\rangle,\\
&\langle \Phi A,\Phi B\rangle=\langle A,-\Phi^{2}B\rangle=\langle A,B\rangle.
\end{align*}
Therefore $E^{(1,0)}$ is isotropic.
Similarly, $E^{(0,1)},L^{\pm},\overline{L^{\pm}}$ are isotropic since $\langle E_{\pm},E_{\pm}\rangle=0$.
\end{proof}

According to \cite{PoWa}, we define
\begin{defn}
Let $(\Phi,E_{\pm})$ be a generalized almost contact structure.
If either of $L^{\pm}$ is Courant involutive, it is called a generalized contact structure.
If both $L^{\pm}$ are Courant involutive, it is called a strong generalized contact structure.
\end{defn}

An almost contact metric structure on $M$ is $(g,\varphi,\xi,\eta)$, where $(\varphi,\xi,\eta)$ is an almost contact structure and $g$ is a Riemannian metric which satisfies
\begin{equation*}
g(\varphi X,\varphi Y)=g(X,Y)-\eta(X)\eta(Y), \qquad \forall X,Y \in TM.
\end{equation*}
We define a generalized almost contact metric structure:

\begin{defn}\label{gacm}
Let $(\Phi,E_{\pm})$ be a generalized almost contact structure.
If $G\colon TM\oplus T^{*}M\rightarrow TM\oplus T^{*}M$ is a generalized Riemannian metric which satisfies
\begin{equation*}
-\Phi G\Phi=G-E_{+}\otimes E_{+}-E_{-}\otimes E_{-},
\end{equation*}
then $(G,\Phi,E_{\pm})$ is a generalized almost contact metric structure.
\end{defn}

This definition satisfies the condition of generalized metric $F$-structure without a signature \cite{Vais3}.

From Definition \ref{gacm} we have that $(G, G\Phi=\Phi G,GE_{\pm}=E_{\mp})$ is also a generalized almost contact metric structure.

\section{generalized Sasakian structure}\label{gss}

There is the intriguing correspondence between the geometry on the cone $C(M)=M\times\mathbb{R}_{>0}$ and the geometry on $M$ \cite{Boy-Gal}.
In fact, an almost contact structure $(\varphi,\xi,\eta)$ gives rises to an almost complex structure $I$ on $C(M)$;
\begin{equation*}
I=\varphi+\eta\otimes\frac{\partial}{\partial t}-dt\otimes\xi,
\end{equation*}
where $e^{t}=r$ denotes the coordinate on $\mathbb{R}_{>0}$.
If $I$ is integrable, an almost contact structure is called a normal almost contact structure.
Let $(\Phi,E_{\pm}=\xi_{\pm}+\eta_{\pm})$ be a generalized almost contact structure on $M$.
we recall a bundle map $\Psi\colon TC(M)\oplus T^{*}C(M)\rightarrow TC(M)\oplus T^{*}C(M)$ by
\begin{align*}
\Psi(E_{+},E_{-})
&=E_{-}\otimes\frac{\partial}{\partial t}-\frac{\partial}{\partial t}\otimes E_{-}+E_{+}\otimes dt -dt\otimes E_{+}\\
&=
\begin{pmatrix}
\eta_{-}\otimes\frac{\partial}{\partial t}-dt\otimes\xi_{+}&\xi_{-}\otimes\frac{\partial}{\partial t}-\frac{\partial}{\partial t}\otimes\xi_{-}\\
\eta_{+}\otimes dt-dt\otimes\eta_{+}&\xi_{+}\otimes dt-\frac{\partial}{\partial t}\otimes\eta_{-}
\end{pmatrix}.
\end{align*}
Then it follows that
\begin{equation*}
\Phi+\Psi(E_{+},E_{-})
\end{equation*}
are generalized almost complex structures on $C(M)$.

\begin{prop}\label{L_t-ver}
There is a one-to-one correspondence between generalized almost contact structures $(\Phi,E_{\pm})$ on $M$ and generalized almost complex structures $\mathcal{J}$ on $C(M)$ such that
\begin{align*}
&\mathcal{L}_{\frac{\partial}{\partial t}}\mathcal{J}=0,\\
&\mathcal{J}\frac{\partial}{\partial t} \in TM\oplus T^{*}M,\qquad
\mathcal{J}dt \in TM\oplus T^{*}M.
\end{align*}
\end{prop}

\begin{proof}
Let $\mathcal{J}$ be a generalized almost complex structure which satisfies above conditions.
Since $\mathcal{J}=-\mathcal{J}^{*}$, if $\mathcal{L}_{\frac{\partial}{\partial t}}\mathcal{J}=0$ then we can write
\begin{equation*}
\mathcal{J}=\mathcal{J}_{M}+A\otimes\frac{\partial}{\partial t}-\frac{\partial}{\partial t}\otimes A+B\otimes dt -dt\otimes B+h\frac{\partial}{\partial t}\otimes dt-hdt\otimes\frac{\partial}{\partial t}
\end{equation*}
where $\mathcal{J}_{M}\colon TM\oplus T^{*}M\rightarrow TM\oplus T^{*}M$, $A,B \in TM\oplus T^{*}M$ and $h \in C^{\infty}(M)$.
From $\mathcal{J}dt \in TM\oplus T^{*}M$, we have $h=0$.
$\mathcal{J}^{2}=-id$ implies that $(\mathcal{J}_{M},B,A)$ is a generalized almost contact structure.

$\Phi+\Psi(E_{+},E_{-})$ is clearly a generalized almost complex structure which satisfies above conditions.
\end{proof}

The integrability condition of $\Phi+\Psi(E_{+},E_{-})$ is given by the following proposition.

\begin{prop}
A generalized almost complex structure $\Phi+\Psi(E_{+},E_{-})$ on $C(M)$ is integrable if and only if a generalized almost contact structure is a strong generalized almost contact structure and $\llbracket E_{+},E_{-}\rrbracket=0$.
\end{prop}

\begin{proof}
Since $E^{(1,0)}$ is $+\sqrt{-1}$-eigenbundle of $\Phi$,
$+\sqrt{-1}$-eigenbundle of $\Phi+\Psi(E_{+},E_{-})$ is generated by
\begin{equation*}
E^{(1,0)},\;E_{+}-\sqrt{-1}\frac{\partial}{\partial t},\;E_{-}-\sqrt{-1}dt.
\end{equation*}
By simple calculations, we have
\begin{align*}
&\left\llbracket X+\alpha,E_{+}-\sqrt{-1}\frac{\partial}{\partial t}\right\rrbracket
=\llbracket X+\alpha,E_{+}\rrbracket\\
&\left\llbracket X+\alpha,E_{-}-\sqrt{-1}dt\right\rrbracket
=\llbracket X+\alpha,E_{-}\rrbracket\\
&\left\llbracket E_{+}-\sqrt{-1}\frac{\partial}{\partial t},E_{-}-\sqrt{-1}dt\right\rrbracket
=\llbracket E_{+},E_{-}\rrbracket,
\end{align*}
where $X+\alpha \in \Gamma(E^{(1,0)})$.
Since $\llbracket E_{+},E_{-}\rrbracket$ is a real section,
$+\sqrt{-1}$-eigenbundle of $\Phi+\Psi(E_{+},E_{-})$ is Courant involutive if and only if both $L^{\pm}$ are Courant involutive and $\llbracket E_{+},E_{-}\rrbracket=0$.
\end{proof}

Let $R$ be an endomorphism of $TM\oplus T^{*}M$ given by
\begin{equation*}
R=
\begin{pmatrix}
r^{-1}&0\\0&r
\end{pmatrix}
=
\begin{pmatrix}
e^{-t}&0\\0&e^{t}
\end{pmatrix}.
\end{equation*}
Then the adjoints
\begin{equation*}
R(\Phi+\Psi(E_{+},E_{-}))R^{-1}
\end{equation*}
are also generalized almost complex structures on $C(M)$.
Let $g$ be a Riemannian metric on $M$.
In Sasakian geometry, the Riemannian cone metric on $C(M)$ is
\begin{equation*}
\tilde{g}=dr^{2}+r^{2}g.
\end{equation*}
Since $R(\Phi+\Psi(E_{+},E_{-}))R^{-1}$ correspond to the cone metric,
$R(\Phi+\Psi(E_{+},E_{-}))R^{-1}$ is more important than $\Phi+\Psi(E_{+},E_{-})$ when we consider about Sasakian structures.

The integrability condition of $R(\Phi+\Psi(E_{+},E_{-}))R^{-1}$ is given by the following theorem.

\begin{thm}\label{int}
A generalized almost complex structure $R(\Phi+\Psi(E_{+},E_{-}))R^{-1}$ on $C(M)$ is integrable if and only if the Nijenhuis operator on $M$ satisfies
\begin{align*}
&\Nij_{M}(A,B,C)\\
&=2\sqrt{-1}\left(\langle E_{-},A\rangle\langle B,C\rangle_{-}
+\langle E_{-},B\rangle\langle C,A\rangle_{-}
+\langle E_{-},C\rangle\langle A,B\rangle_{-}
\right)
\end{align*}
for any $A,B,C \in \Gamma(E^{(1,0)}\oplus L_{E_{+}}\oplus L_{E_{-}})$,
where
\begin{equation*}
\langle X+\alpha,Y+\beta\rangle_{-}=\frac{1}{2}\left(\alpha(Y)-\beta(X)\right)
\end{equation*}
\end{thm}

\begin{proof}
Let $L$ be $+\sqrt{-1}$-eigenbundle of $R(\Phi+\Psi(E_{+},E_{-}))R^{-1}$.
$R(\Phi+\Psi(E_{+},E_{-})R^{-1}$ is integrable if and only if $\Nij_{C(M)}\vert_{L}=0$.
Since the $+\sqrt{-1}$-eigenbundle $L$ is isotropic, $\Nij_{C(M)}\vert_{L}$ is a trilinear operator.
Thus we only need to consider elements in $E^{(1,0)}$, $E_{+}$ and $E_{-}$.
Let $X+\alpha,Y+\beta,Z+\gamma$ be elements of $E^{(1,0)}$.
Then we have from Definition \ref{gac}
\begin{align*}
&\llbracket R(X+\alpha),R(Y+\beta)\rrbracket\\
&=e^{-t}R\llbracket X+\alpha,Y+\beta\rrbracket+(\alpha(Y)-\beta(X))dt.
\end{align*}
Similarly, we have
\begin{align*}
&\left\llbracket R(X+\alpha),R\left(E_{+}-\sqrt{-1}\frac{\partial}{\partial t}\right)\right\rrbracket\\
&=e^{-t}R\llbracket X+\alpha,E_{+}\rrbracket-\sqrt{-1}e^{-2t}X+\sqrt{-1}\alpha+(\alpha(\xi_{+})-\eta_{+}(X))dt,\\
&\llbracket R(X+\alpha),R(E_{-}-\sqrt{-1}dt)\rrbracket\\
&=e^{-t}R\llbracket X+\alpha,E_{-}\rrbracket +(\alpha(\xi_{-})-\eta_{-}(X))dt,\\
&\left\llbracket R\left(E_{+}-\sqrt{-1}\frac{\partial}{\partial t}\right),R(E_{-}-\sqrt{-1}dt)\right\rrbracket\\
&=e^{-t}R\llbracket E_{+},E_{-}\rrbracket +\sqrt{-1}e^{-2t}\xi_{-}-\sqrt{-1}\eta_{-}+(\eta_{+}(\xi_{-}-\eta_{-}(\xi_{+}))dt.
\end{align*}
Then it follows that
\begin{align*}
&\Nij_{C(M)}(R(X+\alpha),R(Y+\beta),R(Z+\gamma))\\
&=e^{-t}\Nij_{M}(X+\alpha,Y+\beta,Z+\gamma).\\
\end{align*}
Similarly, we have
\begin{align*}
&\Nij_{C(M)}\left(R(X+\alpha),R(Y+\beta),R\left(E_{+}-\sqrt{-1}\frac{\partial}{\partial t}\right)\right)\\
&=e^{-t}\Nij_{M}(X+\alpha,Y+\beta,E_{+})+\frac{1}{2}\sqrt{-1}e^{-t}(\beta(X)-\alpha(Y)),\\
&\Nij_{C(M)}\left(R(X+\alpha),R(Y+\beta),R(E_{-}-\sqrt{-1}dt)\right)\\
&=e^{-t}\Nij_{M}(X+\alpha,Y+\beta,E_{-}),\\
&\Nij_{C(M)}\left(R(X+\alpha),R\left(E_{+}-\sqrt{-1}\frac{\partial}{\partial t}\right),R(E_{-}-\sqrt{-1}dt)\right)\\
&=e^{-t}\Nij_{M}(X+\alpha,E_{+},E_{-})-\frac{1}{2}\sqrt{-1}e^{-t}(\eta_{-}(X)-\alpha(\xi_{-})).
\end{align*}
Thus we obtain
\begin{align*}
&\Nij_{C(M)}(A,B,C)\\
&=e^{-t}\Nij_{M}(A,B,C)
-2\sqrt{-1}e^{-t}\langle E_{-},A\rangle\langle B,C\rangle_{-}\\
&\qquad-2\sqrt{-1}e^{-t}\langle E_{-},B\rangle\langle C,A\rangle_{-}
	-2\sqrt{-1}e^{-t}\langle E_{-},C\rangle\langle A,B\rangle_{-}
\end{align*}
for any $A,B,C \in \Gamma(E^{(1,0)}\oplus L_{E_{+}}\oplus L_{E_{-}})$.
Therefore the integrability condition is given by
\begin{align*}
&\Nij_{M}(A,B,C)\\
&=2\sqrt{-1}\left(\langle E_{-},A\rangle\langle B,C\rangle_{-}
+\langle E_{-},B\rangle\langle C,A\rangle_{-}
+\langle E_{-},C\rangle\langle A,B\rangle_{-}
\right)
\end{align*}
for any $A,B,C \in \Gamma(E^{(1,0)}\oplus L_{E_{+}}\oplus L_{E_{-}})$.
\end{proof}

An immediate corollary of Theorem \ref{int} is 

\begin{cor}
Let $(\Phi, E_{\pm})$ be a generalized almost contact structure.
If $R(\Phi+\Psi(E_{+},E_{-}))R^{-1}$ is a generalized complex structure on $C(M)$,
then $E^{(1,0)}\oplus L_{E_{-}}$ is Courant involutive.
Therefore $(\Phi,E_{\pm})$ is a generalized contact structure.
\end{cor}

\begin{proof}
It follows from Theorem \ref{int} that
\begin{align*}
&\Nij_{M}(A,B,C)=0,\qquad  A,B, C \in E^{(1,0)}\oplus L_{E_{-}}.
\end{align*}
Therefore $E^{(1,0)}\oplus L_{E_{-}}$ is Courant involutive.
\end{proof}

\begin{defn}
Let $(\Phi,E_{\pm})$ be a generalized almost contact structure.
If a generalized almost complex structure $R(\Phi+\Psi(E_{+},E_{-}))R^{-1}$ is integrable, a generalized almost contact structure is a called normal generalized almost contact structure.
\end{defn}

Note that this definition differs from a Vaisman's definition \cite{Vais}.

We define a generalized Sasakian structure in terms of a generalized almost contact metric structure.

\begin{defn}\label{gs}
A generalized Sasakian structure on $M$ is a generalized almost contact metric structure $(G,\Phi,E_{\pm})$ such that $R(\Phi+\Psi(E_{+},E_{-}))R^{-1}$ and $R(G\Phi+\Psi(GE_{+},GE_{-}))R^{-1}$ are generalized complex structures on $C(M)$.
\end{defn}

A generalized Sasakian structure $(G,\Phi,E_{\pm})$ on $M$ induces a generalized K\"ahler structure $(R(\Phi+\Psi(E_{+},E_{-}))R^{-1},R(G\Phi+\Psi(GE_{+},GE_{-})R^{-1})$ on $C(M)$.

\begin{rem}\label{b-on-cone}
Definition \ref{gs} coincides with Vaisman's definition in the case of $\kappa=0$ under a modification of degree $r$ \cite{Vais,Vais3} (also see Proposition \ref{L_t-ver}).
The Sasakian structure due to Vaisman allows transformations by $2$-forms $2rdr\wedge\kappa$ $(\kappa \in T^{*}M)$, however the one by our definition does not admit such a $B$-field transformation.
Generalized almost contact structures admit $B$-field transformations by $2$-forms on $M$.
However, Lemma \ref{int-close} and
\begin{equation*}
d(r^{2}\alpha)\neq 0,\qquad \forall\alpha \in \Lambda^{2}T^{*}M
\end{equation*}
show that our definition of generalized Sasakian structures does not admit any $B$-field transformation.
If $\langle \kappa ,E_{\pm}\rangle=0$, there exists a generalized almost contact structure $(\Phi^{\kappa},E_{\pm}^{\kappa})$ such that
\begin{equation*}
\begin{pmatrix}
1&0\\ \frac{2}{r}dr\wedge\kappa&1
\end{pmatrix}
(\Phi+\Psi(E_{+},E_{-}))
\begin{pmatrix}
1&0\\ -\frac{2}{r}dr\wedge\kappa&1
\end{pmatrix}
=\Phi^{\kappa}+\Psi(E_{+}^{\kappa},E_{-}^{\kappa}).
\end{equation*}
However $(G,\Phi^{\kappa},E_{\pm}^{\kappa})$ is not a generalized almost contact metric structure.
More details about transformations by $2$-forms $2rdr\wedge\kappa$ appear in section \ref{ext-gen-con}.
\end{rem}

\begin{exam}
Let $(g,\varphi,\xi,\eta)$ be a Sasakian structure.
If we set
\begin{equation*}
G=
\begin{pmatrix}
0&g^{-1}\\
g&0
\end{pmatrix}
,\quad
\Phi=
\begin{pmatrix}
\varphi&0\\
0&-\varphi^{*}
\end{pmatrix}
,\quad
E_{+}=\xi
,\quad
E_{-}=\eta
\end{equation*}
then $(G,\Phi,E_{\pm})$ becomes a generalized Sasakian structure.
\end{exam}

The next theorem corresponds to Theorem \ref{reconst}.

\begin{thm}\cite{Vais}\label{con-con}
A generalized Sasakian structure on a manifold $M$ is equivalent to a pair $(\varphi_{\pm},\xi_{\pm},\eta_{\pm},g)$ of normal almost contact metric structures with the same metric $g$ which satisfy the following conditions
\begin{align}
&\mathcal{L}_{\xi_{+}}\theta_{+}=-\mathcal{L}_{\xi_{-}}\theta_{-}\label{e-1}\\
&\theta_{\pm}-d\eta_{\pm}+\frac{1}{4}\mathcal{L}_{\xi_{\pm}}\mathcal{L}_{\xi_{\pm}}\theta_{\pm}=0\label{e-2}\\
&d\theta_{\pm}-\eta_{\pm}\wedge \mathcal{L}_{\xi_{\pm}}\theta_{\pm}-\frac{1}{2}\left(d\mathcal{L}_{\xi_{\pm}}\theta_{\pm}\right)^{c_{\pm}}=0,
\end{align}
where $\theta=g(\cdot,\varphi)$ and the upper indices $c_{\pm}$ denote
\begin{equation*}
\alpha^{c_{\pm}}(X_{1},\ldots,X_{k})=\alpha(\varphi_{\pm}X_{1},\ldots,\varphi_{\pm}X_{k}),\qquad \forall \alpha \in \Omega^{k}(M).
\end{equation*}
\end{thm}

Note that a pair of Sasakian structures with the same metric satisfies these conditions.
In the case of a compact connected 3-dimensional manifold, a generalized Sasakian structure is equivalent to a pair of Sasakian structures with the same metric.
In fact, we have

\begin{thm}\label{th:GS on 3-mfds}
Let $M$ be a compact connected 3-dimensional manifold.
Then a pair $(\varphi_{\pm},\xi_{\pm},\allowbreak \eta_{\pm},g)$ of normal almost contact metric structures corresponds to a generalized Sasakian structure if and only if both structures are Sasakian.
\end{thm}

\begin{proof}
A normal almost contact metric structure $(\varphi,\xi,\eta,g)$ is a Sasakian structure if and only if $\theta=d\eta$, where $\theta=g(\cdot,\varphi)$ (c.f. Definition 6.4.4 and Definition 6.5.13 in \cite{Boy-Gal}).
Thus it is sufficient to show that $\theta_{\pm}=d\eta_{\pm}$.
Since $M$ is 3-dimensional, we have
\begin{equation*}
\eta_{+}\wedge d\mathcal{L}_{\xi_{+}}\theta_{+}=0.
\end{equation*}
The inner product by $\xi_{+}$ yields
\begin{equation*}
\eta_{+}\wedge \mathcal{L}_{\xi_{+}}\mathcal{L}_{\xi_{+}}\theta_{+}=d\mathcal{L}_{\xi_{+}}\theta_{+}.
\end{equation*}
From (\ref{e-2}) and Stokes' theorem, we have
\begin{equation*}
0\neq 
\int\eta_{+}\wedge\theta_{+}
=\int\eta_{+}\wedge\left(d\eta_{+}-\frac{1}{4}\mathcal{L}_{\xi_{+}}\mathcal{L}_{\xi_{+}}\theta_{+}\right)
=\int\eta_{+}\wedge d\eta_{+}.
\end{equation*}
Let $U$ be the open set given by
\begin{equation*}
U=\{x \in M\;;\; (\eta_{+}\wedge d\eta_{+})_{x}\neq 0\}.
\end{equation*}
Then $U$ is not empty.
It follows from Darboux's theorem that we have local coordinates $(x,y,z)$ such that
\begin{equation*}
\eta_{+}=dz-ydx,\quad\xi_{+}=\frac{\partial}{\partial z}.
\end{equation*}
Since $\iota_{\xi_{+}}\theta_{+}=0$, there exits a function $f\neq 0$ such that
\begin{equation*}
\theta_{+}=fdx\wedge dy=fd\eta_{+}.
\end{equation*}
From (\ref{e-1}), we have
\begin{equation*}
0=-\iota_{\xi_{-}}\mathcal{L}_{\xi_{-}}\theta_{-}=\iota_{\xi_{-}}\mathcal{L}_{\xi_{+}}\theta_{+}
=\iota_{\xi_{-}}\left(\frac{\partial f}{\partial z}dx\wedge dy\right).
\end{equation*}

Let $V$ be the open set given by
\begin{equation*}
V=\left\{x \in U \;;\;\frac{\partial f}{\partial z}\neq 0\right\}
\end{equation*}
We assume that $V$ is not empty.
Then we have $\xi_{+}=\pm\xi_{-}$ on $V$.
Since $\iota_{\xi_{-}}\theta_{-}=0$, we obtain
\begin{equation*}
\theta_{-}=hd\eta_{-}=\pm hd\eta_{+},
\end{equation*}
where $h$ is a function.
From (\ref{e-1}), we have
\begin{equation*}
\frac{\partial f}{\partial z}d\eta_{+}=-\frac{\partial h}{\partial z}d\eta_{+}.
\end{equation*}
Then, from (\ref{e-2}), we have
\begin{align*}
\left(f-1+\frac{1}{4}\frac{\partial^{2} f}{\partial z^{2}}\right)d\eta_{+}&=0,\\
\pm\left(h-1-\frac{1}{4}\frac{\partial^{2} f}{\partial z^{2}}\right)d\eta_{+}&=0.
\end{align*}
Thus it follows that
\begin{equation*}
f-1=-(h-1).
\end{equation*}
Thus, for $X, Y \in TM$, we obtain
\begin{equation*}
g(Y,\varphi_{-}X)=\theta_{-}(Y,X)=\pm \left(\frac{2}{f}-1\right)\theta_{+}(Y,X)
=g\left(Y,\pm\left(\frac{2}{f}-1\right)\varphi_{+} X\right).
\end{equation*}
Thus it follows that
\begin{equation*}
\varphi_{-}=\pm\left(\frac{2}{f}-1\right)\varphi_{+}.
\end{equation*}
Since $\varphi_{\pm}^{2}=-id+\eta_{\pm}\otimes\xi_{\pm}$, we have $f=1$.
However this is a contradiction because $\frac{\partial f}{\partial z}\neq 0$.
Therefore $\frac{\partial f}{\partial z}=0$ on $U$, we have $L_{\xi_{+}}\theta_{+}=0$ and $\theta_{\pm}=d\eta_{\pm}$ on $U$.
Since $\eta_{+}\wedge\theta_{+}\neq 0$ on $M$,
we have $\overline{U}\subset U$.
Since $M$ is connected and $U$ is not empty,
we get $U=M$ and $\theta_{\pm}=d\eta_{\pm}$ on $M$.
\end{proof}

On a compact 3-dimensional manifold, a generalized Sasakian structure is equivalent to a pair of Sasakian structures.
However, there exists a non-compact example which is not a pair of Sasakian structures.

\begin{exam}\label{ex:GS, non-Sasakian pair}
Let $(M',g',J',\omega')$ be a K\"ahler manifold and $M=M'\times (0,\pi/2)$.
To construct normal almost contact metric structures, we define
\begin{align*}
\varphi&=J',\quad \xi=\frac{\partial}{\partial z},\quad \eta=dz\\
g&=\sin(2z)g'+dz\otimes dz,
\end{align*}
where $z$ denotes the coordinate on $(0,\pi/2)$.
Then $(g,\pm\varphi,\xi,\eta)$ are normal almost contact metric structures but not Sasakian structures.

On $C(M)=M'\times (0,\pi/2)\times\mathbb{R}_{>0}$, we define complex structures and a metric by
\begin{align*}
J_{\pm}&=\pm\varphi-\frac{1}{r}dr\otimes\frac{\partial}{\partial z} +dz\otimes r\frac{\partial}{\partial r},\\
\tilde{g}&=r^{2}g+dr\otimes dr.
\end{align*}
Then $(\tilde{g},J_{\pm})$ is a bi-Hermitian structure and
\begin{align*}
\omega_{\pm}&=\tilde{g}(\cdot,J_{\pm}\cdot)=\pm r^{2}\sin(2z)\omega'+2rdr\wedge dz,\\
d\omega_{\pm}&=\pm2r\sin(2z)dr\wedge \omega'\pm 2r^{2}\cos(2z)dz\wedge \omega'.
\end{align*}
Thus
\begin{align*}
&d\omega_{\pm}(J_{\pm}\cdot,J_{\pm}\cdot,J_{\pm}\cdot)\\
&=\pm 2r\sin(2z) (rdz)\wedge\omega'\pm2r^{2}\cos(2z)\left(-\frac{1}{r}dr\right)\wedge\omega'\\
&=\pm2r^{2}\sin(2z)dz\wedge \omega'\mp2r\cos(2z)dr\wedge \omega'\\
&=\pm d(-r^{2}\cos(2z)\omega').
\end{align*}
Therefore $(\tilde{g},-r^{2}\cos(2z)\omega',J_{\pm})$ is a generalized K\"ahler structure and induces a generalized Sasakian structure.
If we set $\rho=(\omega')^{-1}$ on $M'$, we have
\begin{gather*}
G=
\begin{pmatrix}
1&0\\
-\cos(2z)\omega'&1
\end{pmatrix}
\begin{pmatrix}
0&g^{-1}\\
g&0
\end{pmatrix}
\begin{pmatrix}
1&0\\
\cos(2z)\omega'&1
\end{pmatrix},
\\
\Phi=
\begin{pmatrix}
1&0\\
-\cos(2z)\omega'&1
\end{pmatrix}
\begin{pmatrix}
0&\frac{1}{\sin(2z)}\rho\\
-\sin(2z)\omega'&0
\end{pmatrix}
\begin{pmatrix}
1&0\\
\cos(2z)\omega'&1
\end{pmatrix},\\
E_{+}=
\begin{pmatrix}
1&0\\
-\cos(2z)\omega'&1
\end{pmatrix}
\begin{pmatrix}
\frac{\partial}{\partial z}\\0
\end{pmatrix}
,\quad
E_{-}=
\begin{pmatrix}
1&0\\
-\cos(2z)\omega'&1
\end{pmatrix}
\begin{pmatrix}
0\\ dz
\end{pmatrix}.
\end{gather*}
\end{exam}

\section{an extension of generalized almost contact structure}\label{ext-gen-con}

In this section, we consider about a transformations by $2$-forms $2rdr\wedge\kappa$ $(\kappa \in T^{*}M)$ in Remark \ref{b-on-cone}.
By a direct calculation, we have
\begin{align*}
&\begin{pmatrix}
1&0\\\frac{2}{r}dr\wedge\kappa&1
\end{pmatrix}
(\Phi+\Psi(E_{+},E_{-}))
\begin{pmatrix}
1&0\\-\frac{2}{r}dr\wedge\kappa&1
\end{pmatrix}\\
&=\Phi-\kappa\otimes E_{-}+E_{-}\otimes\kappa\\
&\qquad +E_{-}\otimes \frac{\partial}{\partial t}-\frac{\partial}{\partial t}\otimes E_{-}\\
&\qquad +(E_{+}+\Phi\kappa+2\langle E_{-},\kappa\rangle\kappa)\otimes dt-dt\otimes (E_{+}+\Phi\kappa+2\langle E_{-},\kappa\rangle\kappa)\\
&\qquad +2\langle E_{-},\kappa\rangle\frac{\partial}{\partial t}\otimes dt-2\langle E_{-},\kappa\rangle dt\otimes \frac{\partial}{\partial t}.
\end{align*}
We extend a definition of generalized almost contact structure to admit above transformations.
\begin{defn}\label{gfacs}
A generalized $f$-almost contact structure on smooth manifold $M$ is a 4-tuple $(\Phi^{f},E_{+}^{f},E_{-}^{f},f)$, where $\Phi^{f}$ is an endomorphism of $TM\oplus T^{*}M$, $E_{\pm}^{f}$ are sections of $TM\oplus T^{*}M$ and $f$ is a $C^{\infty}$- function on $M$ which satisfy
\begin{align*}
&\Phi^{f}+(\Phi^{f})^{*}=0\\
&(\Phi^{f})^{2}=-id+E_{+}^{f}\otimes E_{-}^{f}+E_{-}^{f}\otimes E_{+}^{f},\\
&\Phi^{f} E_{+}^{f}=fE_{+}^{f},\qquad \Phi^{f} E_{-}^{f}=-fE_{-}^{f},\\
&2\langle E_{+}^{f},E_{-}^{f}\rangle=1+f^{2},\qquad \langle E_{\pm}^{f},E_{\pm}^{f}\rangle=0.
\end{align*}
\end{defn}

A definition of generalized almost contact structure is the special case of $f=0$.
We already have an example. 

\begin{exam}\label{constE-}
Let $(\Phi,E_{\pm})$ be a generalized almost contact structure and $\kappa \in T^{*}M$.
Then
$(\Phi-\kappa\otimes E_{-}+E_{-}\otimes\kappa,E_{+}+\Phi\kappa+2\langle E_{-},\kappa\rangle\kappa,E_{-},2\langle E_{-},\kappa\rangle)$
is a generalized $f$-almost contact structure.
\end{exam}

By an analogue to this example, we have

\begin{exam}\label{constE+}
Let $(\Phi,E_{\pm})$ be a generalized almost contact structure and $\kappa \in T^{*}M$.
Then
$(\Phi-\kappa\otimes E_{+}+E_{+}\otimes\kappa,E_{+},E_{-}+\Phi\kappa+2\langle E_{+},\kappa\rangle\kappa,-2\langle E_{+},\kappa\rangle)$
is a generalized $f$-almost contact structure.
\end{exam}

We generalize methods for constructing these examples.

\begin{defn}
Let $(\Phi^{f},E^{f}_{\pm},f)$ be a generalized $f$-almost contact structure and $\kappa \in T^{*}M$.
We define $\mathcal{K}_{-}(\kappa)$-deformation by
\begin{align*}
\mathcal{K}_{-}(\kappa)
\begin{pmatrix}
\Phi^{f},\\E^{f}_{+},\\E^{f}_{-},\\f
\end{pmatrix}=
\begin{pmatrix}
\Phi^{f}-\kappa\otimes E^{f}_{-}+E^{f}_{-}\otimes\kappa,\\
E^{f}_{+}+\Phi^{f}\kappa+2\langle E^{f}_{-},\kappa\rangle\kappa+f\kappa,\\
E^{f}_{-},\\
f+2\langle E^{f}_{-},\kappa\rangle
\end{pmatrix}
\end{align*}
and $\mathcal{K}_{+}(\kappa)$-deformation by
\begin{align*}
\mathcal{K}_{+}(\kappa)
\begin{pmatrix}
\Phi^{f},\\E^{f}_{+},\\E^{f}_{-},\\f
\end{pmatrix}=
\begin{pmatrix}
\Phi^{f}-\kappa\otimes E^{f}_{+}+E^{f}_{+}\otimes\kappa,\\
E^{f}_{+},\\
E^{f}_{-}+\Phi^{f}\kappa+2\langle E^{f}_{+},\kappa\rangle\kappa-f\kappa,\\
f-2\langle E^{f}_{+},\kappa\rangle
\end{pmatrix}.
\end{align*}
\end{defn}

By a simple calculation, we get

\begin{lem}\label{K-Bcom}
$\mathcal{K}_{\pm}(\kappa)$-deformations and $B$-field transformations are commutative.
\end{lem}

Also, we get

\begin{lem}\label{comK+-}
Let $(\Phi^{f},E^{f}_{\pm},f)$ be a generalized $f$-almost contact structure and $\alpha,\beta \in T^{*}M$.
Then we have
\begin{equation*}
\mathcal{K}_{\pm}(\beta)\left(\mathcal{K}_{\pm}(\alpha)(\Phi^{f},E^{f}_{\pm},f)\right)=\mathcal{K}_{\pm}(\alpha+\beta)(\Phi^{f},E^{f}_{\pm},f).
\end{equation*}
\end{lem}

Note that $\mathcal{K}_{+}(\alpha)$-deformations and $\mathcal{K}_{-}(\beta)$-deformations are not commutative.

We can construct any generalized $f$-almost contact structure from a generalized almost contact structure with $\mathcal{K}_{\pm}(\kappa)$-deformations.

\begin{prop}
For any generalized $f$-almost contact structure $(\Phi^{f},E^{f}_{\pm},f)$, there exists a generalized almost contact structure $(\Phi,E_{\pm})$ and 1-forms $\alpha,\beta \in T^{*}M$ such that
\begin{equation*}
(\Phi,E_{\pm},0)=\mathcal{K}_{-}(\beta)\left(\mathcal{K}_{+}(\alpha)(\Phi^{f},E^{f}_{\pm},f)\right).
\end{equation*}
\end{prop}

\begin{proof}
From a direct calculation, the fourth element of the generalized $f$-almost contact structure is
\begin{align*}
&f-2\langle E_{+}^{f},\alpha\rangle+2\langle E_{-}^{f}+\Phi^{f}\alpha+2\langle E_{+}^{F}\alpha\rangle \alpha -f\alpha,\beta\rangle\\
&=f-2\langle E_{+}^{f},\alpha\rangle+2\langle E_{-}^{f},\beta\rangle+\langle\Phi^{f}\alpha,\beta\rangle.
\end{align*}
If we assume that $\beta=-\alpha$, then we have
\begin{equation*}
f-2\langle E_{+}^{f}+E_{-}^{f},\alpha\rangle.
\end{equation*}
Since $\langle E_{+}^{f},E_{-}^{f}\rangle \neq 0$, there exists a 1-form $\alpha$ such that $2\langle E^{f}_{+}+E^{f}_{-},\alpha\rangle=f$.
\end{proof}

By a similar argument in section \ref{gss}, we define a bundle map $\Psi^{f} \colon TC(M)\oplus T^{*}C(M)\rightarrow TC(M)\oplus T^{*}C(M)$ by
\begin{align*}
\Psi^{f}(E^{f}_{+},E^{f}_{-})
&=E_{-}^{f}\otimes r\frac{\partial}{\partial r}-r\frac{\partial}{\partial r}\otimes E_{-}^{f}+E_{+}^{f}\otimes\frac{1}{r}dr\otimes-\frac{1}{r}dr\otimes E_{+}^{f}+f\frac{\partial}{\partial r}\otimes dr-f dr\otimes \frac{\partial}{\partial r}.
\end{align*}
Then we have two generalized almost complex structures $\mathcal{I}'$, $\mathcal{I}$ on $C(M)$,
\begin{align*}
\mathcal{I}'(\Phi^{f},E^{f}_{\pm},f)=\Phi^{f}+\Psi^{f}(E_{+}^{f},E_{-}^{f})
\end{align*}
and 
\begin{align*}
\mathcal{I}(\Phi^{f},E^{f}_{\pm},f)=R(\Phi^{f}+\Psi^{f}(E^{f}_{+},E^{f}_{-}))R^{-1}.
\end{align*}

From a proof of Proposition \ref{L_t-ver}, we have next proposition.

\begin{prop}\label{L_t}
There is a one-to-one correspondence between generalized $f$-almost contact structures $(\Phi^{f},E^{f}_{\pm},f)$ on $M$ and generalized almost complex structures $\mathcal{J}$ on $C(M)$ such that
\begin{align*}
&\mathcal{L}_{\frac{\partial}{\partial t}}\mathcal{J}=0.
\end{align*}
\end{prop}

The construction of $\mathcal{K}_{-}(\kappa)$-deformation shows
\begin{lem}
Let $(\Phi^{f},E^{f}_{\pm},f)$ be a generalized $f$-almost contact structure and $\kappa \in T^{*}M$.
Then we have
\begin{equation*}
\begin{pmatrix}
1&0\\ 2rdr\wedge\kappa&1
\end{pmatrix}
\mathcal{I}(\Phi^{f},E_{\pm}^{f},f)
\begin{pmatrix}
1&0\\ -2rdr\wedge\kappa&1
\end{pmatrix}
=\mathcal{I}(\mathcal{K}_{-}(\kappa)(\Phi^{f},E_{\pm}^{f},f)).
\end{equation*}
\end{lem}

On the other hand, $\mathcal{K}_{+}(\kappa)$-deformation means next theorem.

\begin{thm}\label{Ka-ga}
Let $g$ be a Riemannian metric on $M$ and $\alpha \in T^{*}M$.
Then
\begin{equation*}
G=
\begin{pmatrix}
0&g^{-1}\\ g&0
\end{pmatrix}
\end{equation*}
is a generalized Riemannian metric on $M$ and
\begin{equation*}
\tilde{G}_{\alpha}
=
G+
\begin{pmatrix}
0&\substack{-g^{-1}\alpha\otimes\frac{\partial}{\partial t}-\frac{\partial}{\partial t}\otimes g^{-1}\alpha\\
+(1+g^{-1}(\alpha,\alpha))\frac{\partial}{\partial t}\otimes\frac{\partial}{\partial t}}\\
(\alpha+dt)\otimes(\alpha+dt)&0
\end{pmatrix}
\end{equation*}
is a generalized Reimannian metric on $C(M)$.
A generalized almost complex structure $\mathcal{I}'$ constructed by a generalized $f$-structure $(\Phi^{f},E^{f}_{\pm},f)$ satisfies
\begin{equation*}
\tilde{G}_{\alpha}=-\mathcal{I}'\tilde{G}_{\alpha}\mathcal{I}'
=-(\Phi^{f}+\Psi^{f}(E_{+}^{f},E_{-}^{f}))\tilde{G}_{\alpha}(\Phi^{f}+\Psi^{f}(E_{+}^{f},E_{-}^{f}))
\end{equation*}
if and only if
there exists a generalized almost contact metric structure $(G,\Phi,E_{\pm})$ such that
\begin{equation*}
(\Phi^{f},E^{f}_{\pm},f)=\mathcal{K}_{+}(\alpha)(\Phi,E_{\pm},0).
\end{equation*}
\end{thm}

\begin{proof}
From $\tilde{G}_{\alpha}=-\mathcal{I}'\tilde{G}_{\alpha}\mathcal{I}'$, we get next equations
\begin{equation}\label{ga-m}
\begin{split}
&-G-\alpha\otimes\alpha\\
&=
\Phi^{f}G\Phi^{f}-\Phi^{f}\alpha\otimes\Phi^{f}\alpha+E_{-}^{f}\otimes\Phi^{f}\alpha-E_{+}^{f}\otimes\Phi^{f}G_{\alpha}+\Phi^{f}\alpha\otimes E_{-}^{f}\\
&\quad-E_{-}^{f}\otimes E_{-}^{f}-\Phi^{f}G\alpha\otimes E_{+}^{f}-(1+g^{-1}(\alpha,\alpha))E_{+}^{f}\otimes E_{+}^{f}
\end{split}
\end{equation}
\begin{equation}\label{ga-mdt}
\begin{split}
&\alpha\\
&=\Phi^{f}GE_{+}^{f}+2\langle E_{+}^{f},\alpha\rangle\Phi^{f}\alpha-2\langle E_{+}^{f},\alpha\rangle E_{-}^{f}+2\langle E_{+}^{f},G\alpha\rangle E_{+}^{f}\\
&\quad +f\Phi^{f}\alpha-fE_{-}^{f}
\end{split}
\end{equation}
\begin{equation}\label{ga-ddt}
\langle E_{-}^{f},GE_{+}^{f}\rangle+2\langle\alpha,E_{+}^{f}\rangle\langle\alpha,E_{-}^{f}\rangle+f\langle E_{-}^{f},\alpha\rangle+f\langle G\alpha,E_{+}^{f}\rangle=0
\end{equation}
\begin{equation}\label{ga-dt}
2\langle E_{+}^{f},GE_{+}^{f}\rangle+4\langle\alpha,E_{+}^{f}\rangle^{2}+2f\langle E_{+}^{f},\alpha\rangle+2f\langle\alpha,E_{+}^{f}\rangle+f^{2}=1
\end{equation}

From $\langle$(\ref{ga-mdt})$,E^{f}_{-}\rangle-f$(\ref{ga-ddt}), we have
\begin{equation}\label{ga-ge+e-}
\langle E_{+}^{f},G\alpha\rangle=\langle\alpha,E_{-}^{f}\rangle.
\end{equation}
From $\langle$(\ref{ga-mdt})$,E^{f}_{+}\rangle+f/2$(\ref{ga-dt}), we have
\begin{equation}\label{ga-e+a}
2\langle E_{+}^{f},\alpha\rangle=-f.
\end{equation}
We substitute (\ref{ga-ge+e-}) and (\ref{ga-e+a}) into (\ref{ga-ddt}) and (\ref{ga-mdt}), then we have
\begin{equation}\label{ga-G+}
\langle E_{-}^{f},GE_{+}^{f}\rangle=-f\langle G\alpha,E_{+}^{f}\rangle=-f\langle\alpha,E_{-}^{f}\rangle.
\end{equation}
\begin{equation}\label{ga-alpha}
\begin{split}
\alpha&=
\Phi^{f}GE_{+}^{f}-f\Phi^{f}\alpha+fE_{-}^{f}+2\langle E_{-}^{f},\alpha\rangle E_{+}^{f}+f\Phi^{f}\alpha-fE_{-}^{f}\\
&=\Phi^{f}GE_{+}^{f}+2\langle E_{-}^{f},\alpha\rangle E_{+}^{f}
\end{split}
\end{equation}
From (\ref{ga-G+}) and (\ref{ga-alpha}), we get
\begin{equation}\label{ga-phialpha}
\Phi^{f}\alpha=-GE_{+}^{f}+E_{-}^{f}.
\end{equation}
(\ref{ga-m}), (\ref{ga-e+a}) and (\ref{ga-phialpha}) shows that
\begin{equation*}
\mathcal{K}_{+}(-\alpha)(\Phi^{f},E^{f}_{\pm},f)
\end{equation*}
and $G$ become a generalized almost contact metric structure.
By Lemma \ref{comK+-}, we have a generalized almost contact metric structure $(G,\Phi,E_{\pm})$ such that
\begin{equation*}
(\Phi^{f},E^{f}_{\pm},f)=\mathcal{K}_{+}(\alpha)(\Phi,E_{\pm},0),
\end{equation*}
and vice versa.
\end{proof}

By Lemma \ref{K-Bcom} and Theorem \ref{Ka-ga}, we get

\begin{cor}
Let $(G,\Phi,E_{\pm})$ be a generalized almost contact metric structure.
Since there is a two-form $B\in \wedge^{2}T^{*}M$ and a Riemannian metric $g$ such that
\begin{equation*}
G=
e^{B}
\begin{pmatrix}
0&g^{-1}\\ g&0
\end{pmatrix}
e^{-B},
\end{equation*}
two generalized almost complex structures $\mathcal{I}'(\mathcal{K}_{+}(\alpha)(\Phi,E_{\pm},0))$ and $\mathcal{I}'(\mathcal{K}_{+}(\alpha)(G\Phi,GE_{\pm},0))$ are commutative and
\begin{equation*} 
-\mathcal{I}'(\mathcal{K}_{+}(\alpha)(\Phi,E_{\pm},0))\circ \mathcal{I}'(\mathcal{K}_{+}(\alpha)(G\Phi,GE_{\pm},0))
\end{equation*}
is a generalized Riemannian metric on $C(M)$.
\end{cor}

Now we can define a generalized $f$-almost contact metric structure and a generalized $f$-Sasakian structure.

\begin{defn}\label{f-gacm}
Let $(\Phi^{f},E^{f}_{\pm},f)$ be a generalized $f$-almost contact structure.
If there exists a generalized almost contact metric structure $(G,\Phi,E_{\pm})$ and two $1$-forms $\alpha,\beta \in T^{*}M$ such that 
\begin{equation*}
(\Phi^{f},E^{f}_{\pm},f)=
\mathcal{K}_{-}(\beta)\left(\mathcal{K}_{+}(\alpha)(\Phi,E_{\pm},0)\right),
\end{equation*}
then we called $(G,\alpha,\beta,\Phi^{f},E^{f}_{\pm},f)$ a generalized $f$-almost contact metric structure.
\end{defn}

\begin{defn}\label{f-gs}
Let $(G,\alpha,\beta,\Phi^{f},E^{f}_{\pm},f)$ be a generalized $f$-almost contact metric structure.
Then there is a generalized almost contact metric structure $(G,\Phi,E_{\pm})$ such that
\begin{equation*}(\Phi^{f},E^{f}_{\pm},f)=
\mathcal{K}_{-}(\beta)\circ\mathcal{K}_{+}(\alpha)(\Phi,E_{\pm},0).
\end{equation*}
If two generalized almost complex structures
\begin{align*}
&\mathcal{I}(\mathcal{K}_{-}(\beta)\circ\mathcal{K}_{+}(\alpha)(\Phi,E_{\pm},0))=\mathcal{I}(\Phi^{f},E_{\pm}^{f},f)\\
&\mathcal{I}(\mathcal{K}_{-}(\beta)\circ\mathcal{K}_{+}(\alpha)(G\Phi,GE_{\pm},0))
\end{align*}
are integral,
we called $(G,\alpha,\beta,\Phi^{f},E^{f}_{\pm},f)$ a generalized $f$-Sasakian structure.
\end{defn}

\begin{rem}
From the view point of Definition \ref{f-gs}, we can express that 
Definition \ref{gs} is the special case of $\alpha=\beta=0$.
The Sasakian structure due to Vaisman is the special case of $\alpha=0$.
\end{rem}

On compact manifolds which have more than five dimensions, we don't know whether there are examples of generalized Sasakian structures which are not pairs of Sasakian structures.
On the other hand, we can get non-trivial examples of generalized $f$-Sasakian structures by transformations of $3$-Sasakian structures which is similar to get non-trivial examples of generalized K\"ahler structures by transformations of hyperk\"ahler structures.
However, we need a notion of ``ordinary'' $f$-almost contact structure to study a generalized $f$-almost contact structure in detail.


\begin{thebibliography}{10}
\bibitem{Apo-Gau}
V. Apostolov, P. Gauduchon, G. Grantcharov, {\it Bi-Hermitian structures on complex surfaces}, Proc. London Math. Soc. (3) 79 (1999), no. 2, 414-428, 
Corrigendum: Proc. London Math. Soc. (3) 92 (2006), no. 1, 200-202. 
\bibitem{Boy-Gal}C. Boyer and K. Galicki, {\it Sasakian Geometry}, Oxford university press, 2008
\bibitem{Fu-Po}
A. Fujiki, M. Pontecorvo, {\it Anti-self-dual bihermitian structures on Inoue surfaces}, J. Differential Geom. 85 (2010), no. 1, 15-71.
\bibitem{Gualt-d}M. Gualtieri {\it Generalized complex geometry}, PhD thesis, Oxford University, 2003. arXiv:math.DG/0401221.
\bibitem{Gualt-d-p}M. Gualtieri {\it Generalized complex geometry}, Ann. of Math. (2) 174 (2011), no.1, 75-123.
\bibitem{Hit}N. J. Hitchin {\it Generalized Calabi-Yau manifolds}, Quart. J. Math. 54 (2003), 281-308.
\bibitem{PoWa}Y. S. Poon and A. Wade. {\it Generalized contact structures}, J. Lond. Math. Soc. (2) 83 (2011), no.2, 333-352.
\bibitem{Vais}Izu Vaisman. {\it From generalized K\"ahler to generalized Sasakian structure}, J. Geom. Symmetry Phys. 18 (2010), 63-86
\bibitem{Vais3}Izu Vaisman. {\it Generalized CRF-structures}, Geom Dedicata 133 (2008), 129-154.
\end{thebibliography}

\end{document}